\documentclass[12pt,oneside,a4paper]{amsart}

\usepackage[T1]{fontenc}
\usepackage{ae}
\usepackage[english]{babel}
\usepackage{amsmath}
\usepackage{amssymb}
\usepackage{amsthm}
\usepackage{mathrsfs}
\usepackage{enumitem}
\usepackage[foot]{amsaddr}

\usepackage[utf8]{inputenc}  
\usepackage{amscd}
\usepackage{mathrsfs}
\usepackage[all]{xy}
\usepackage[OT1]{fontenc}
\usepackage{amssymb,latexsym}
\usepackage{type1cm,ae}
\usepackage{bbding}
\usepackage[dvipsnames]{xcolor}
\usepackage{pgf,tikz}
\usepackage{changepage}
\usetikzlibrary[arrows]
\usetikzlibrary{decorations.pathreplacing}

\usepackage{afterpage}
\usepackage[doi=false,isbn=false,url=false,backend=biber,style=alphabetic,maxbibnames=99,minalphanames=5,maxalphanames=10,giveninits=true]{biblatex}

\DeclareFieldFormat{pages}{#1}
\renewbibmacro{in:}{%
	\ifentrytype{article}
	{}
	{\bibstring{in}%
		\printunit{\intitlepunct}}}
\AtEveryBibitem{
	\ifentrytype{book}{%
		\clearfield{pages}%
	}{%
	}%
}

\usepackage{graphicx}
\usepackage{xspace}
\usepackage{setspace}
\usepackage[english]{babel}
\usepackage{tikz}
\usepackage{etoolbox}
\usepackage[hidelinks,linktocpage,colorlinks=true,linkcolor=blue,citecolor=red]{hyperref}

\newcommand{\e}{\varepsilon}

\newcommand{\va}{\enspace}
\newcommand{\sumn}{\sum_{i=1}^{n}}
\newcommand{\sumjk}{\sum_{j,k=1}^{n}}

\newcommand{\Om}{\Omega}
\newcommand{\N}{\mathbb{N}}
\newcommand{\Rnn}{\mathbb{R}^{n}}
\newcommand{\R}{\mathbb{R}}

\newcommand{\Z}{\mathbb{Z}}
\newcommand{\chr}[2]{\Gamma^{#1}_{\hphantom{#1}#2}}
\DeclareMathOperator{\spt}{spt}

\DeclareMathOperator{\dive}{div}

\DeclareMathOperator{\Tr}{Tr}

\makeatletter
\renewenvironment{proof}[1][\proofname]{%
	\par\pushQED{\qed}\normalfont%
	\topsep6\p@\@plus6\p@\relax
	\trivlist\item[\hskip\labelsep\bfseries#1\@addpunct{.}]%
	\ignorespaces
}{%
	\popQED\endtrivlist\@endpefalse
}
\makeatother

\newtheorem{thm}{Theorem}[section]
\newtheorem{prop}[thm]{Proposition}

\newtheorem{cor}[thm]{Corollary}

\theoremstyle{definition}

\numberwithin{equation}{section}

\author{Janne Nurminen}
\address{Department of Mathematics and Statistics, University of Jyväskylä}
\email{janne.s.nurminen@jyu.fi}

\begin{document}
	
	\title{AN INVERSE PROBLEM FOR THE MINIMAL SURFACE EQUATION}

	\begin{abstract}
		We use the method of higher order linearization to study an inverse boundary value problem for the minimal surface equation on a Riemannian manifold $(\Rnn,g)$, where the metric $g$ is conformally Euclidean. In particular we show that with the knowledge of Dirichlet-to-Neumann map associated to the minimal surface equation, one can determine the Taylor series of the conformal factor $c(x)$ at $x_n=0$ up to a multiplicative constant. We show this both in the full data case and in some partial data cases.
		
		\vspace{5pt}
	\end{abstract}
	\maketitle
	
	\tableofcontents

	
	\section{Introduction}
	
	This article focuses on an inverse problem for the minimal surface equation (MSE), which is a quasilinear elliptic PDE. In particular we consider MSE on a manifold $(\Rnn,g), n\geq3,$ where $g_{ij}(x)=c(x)\delta_{ij},$ $1\leq i,j\leq n,$ with $c\in C^{\infty}(\Rnn),$ $c(x)>0$ for all $x\in\Rnn$, that is the metric is conformally Euclidean. The aim is to use the method of higher order linearization to recover information about the conformal factor $c$ from boundary measurements. This method, which uses the nonlinearity of the partial differential equation as a tool, was first introduced in \cite{MR3802298}  in the case of a nonlinear wave equation and was further developed in \cite{MR4188325}, \cite{MR4104456} for nonlinear elliptic equations. 
	
	The novelty of this work is that we use higher order linearization in the case of MSE. For a sufficiently smooth function $u:\Om\subset\R^{n-1}\to\R$, $\Om$ a bounded domain with $C^{\infty}$ boundary, consider $\text{Graph}_u:=\{(x',u(x')): x'\in\Om\}\subset\Rnn$. If $c\equiv1$ we would call $\text{Graph}_u$ a minimal surface if and only if the function $u$ solves the Euclidean MSE
	\[\dive\left(\frac{\nabla u}{\sqrt{1+|\nabla u|^2}}\right)=0,\quad \text{in}\va\Om.\] 
	Define then a function $F\colon\R^{n^2}\to\R$,
	\begin{align}\label{isoF}
		F(x',u,p,P):=&\:-\sum_{i=1}^{n-1}P_{ii} - \frac{n-1}{2c(x',u)}\left(\sum_{i=1}^{n-1}p_i\partial_{x_i}c(x',u)-\partial_{x_n}c(x',u)\right) \\\notag
		&+ \frac{1}{1+|p|^2}\sum_{i,j=1}^{n-1}P_{ij}p_ip_j,\notag
	\end{align}
	where $p=(p_1,\ldots,p_{n-1})\in\R^{n-1}, P=(P_{ij})$ is an $(n-1)\times(n-1)$ matrix and $x'\in\R^{n-1}$. With the conformally Euclidean metric, MSE takes the form
	\begin{align}\label{minsurfeq}
		F(x',u,\nabla u,\nabla^2u)&=-\Delta u-\frac{n-1}{2c}\left(\nabla_{x'}c\cdot\nabla u-\partial_{x_n}c\right)+\frac{\nabla u^T\nabla^2u\nabla u}{1+|\nabla u|^2}\\\notag
		&=-\dive_{g_{n-1}}\left(\frac{\nabla u}{(1+|\nabla u|^2)^{1/2}}\right)+\frac{(n-1)\partial_{x_n}c}{2c(1+|\nabla u|^2)^{1/2}}\\\notag
		&=0.
	\end{align}
	for $x'\in\Om$. Here $\dive_{g_{n-1}}(a)=\sum_{i=1}^{n-1}\left(\partial_{x_i}a_i+\sum_{j=1}^{n-1}a_j\chr{i}{ij}\right)$ is the Riemannian divergence with respect to the first $n-1$ variables and $\chr{i}{ij}$ is the Christoffel symbol corresponding to the metric $g$. The derivation of this equation is done in Section \ref{section derivation}.

	In this work we consider a boundary value problem
	\begin{equation*}
		\left\{\begin{array}{ll}
			F(x',u,\nabla u,\nabla^2u)=0 & \text{in}\,\, \Om \\
			u=f, & \text{on}\,\, \partial\Om,
		\end{array} \right.
	\end{equation*}
	and prove that it is well-posed (Section \ref{section well-posedness}) for a certain class of small boundary values $f$. To be more precise, we show that there is $\delta>0$ such that whenever $f\in C^s(\partial\Om)$, $s>3$, $s\notin\N,$ with $||f||_{C^s(\partial\Om)}\leq\delta$, there exists a unique small solution $u\in C^s(\bar{\Om})$ with sufficiently small norm. Let $U_{\delta}:=\{h\in C^s(\partial\Om):||h||_{C^s(\partial\Om)}<\delta\}$. Thus the Dirichlet-to-Neumann (DN) map can now be defined for these small solutions as
	\begin{equation}\label{DN minimal}
		\Lambda_c\colon U_{\delta}\to C^{s-1}(\partial\Om),\quad f\mapsto \partial_{\nu}u_f\big|_{\partial\Om}.
	\end{equation}
	Here $C^s=C^{k,\alpha}$, $k\in\Z$, $0<\alpha<1$, is the standard Hölder space (see for example \cite[Section 5.1]{MR2597943}) and $\partial_{\nu}u_f$ is the Euclidean boundary normal derivative. One can think of the normal derivative on the boundary as tension on the boundary caused by the minimal surface. From the knowledge of the DN map, can we recover information about the metric $g$?
	
	It is worth noting that there is a small gauge invariance for the equation \eqref{minsurfeq} and thus for the DN map. That is, if you instead of $c$ put $\lambda c$, $\lambda\neq0$, into \eqref{minsurfeq}, the equation stays the same. Thus also the DN maps $\Lambda_c$ and $\Lambda_{\lambda c}$ are the same.
	
	We also consider partial data cases, that is, if we have knowledge of the DN map in an open subset $\Gamma$ of the boundary $\partial\Om$. In this case the partial DN map is defined for $f\in U_{\delta}$, $\spt(f)\subset\Gamma$, as
	\begin{equation}\label{DN minimal partial}
		\Lambda^{\Gamma}_c\colon U_{\delta}\to C^{s-1}(\partial\Om),\quad f\mapsto \partial_{\nu}u_f\big|_{\Gamma}.
	\end{equation}
	Can we recover information about the metric $g$ if we have knowledge of this partial DN map?
	
	These are our inverse problems for the MSE and our main result gives the following answers. Before stating it, we denote by $F^j$ the function $F$ with $c$ replaced by $c_j$.
	
	
	\begin{thm}\label{main thm}
		Let $\Om\subset\R^{n-1}$, $n\geq3$, be a bounded domain with $C^{\infty}$ boundary, $(\R^n,g_1)$, $(\R^n,g_2)$ be two Riemannian manifolds with $(g_j)_{ik}(x)=c_j(x)\delta_{ik},$ where $c_j\in C^{\infty}(\R^n),$ $c_j(x)>0$ for $j=1,2$ and for all $x\in\R^n$. Assume that $\partial_{x_n}c_j(x',0)=\partial_{x_n}^2c_j(x',0)=0$ for $x'\in\Om$. We have four cases:
		\begin{enumerate}
			\item Let $n>3$ and $\Lambda_{c_j}$ be the DN maps associated to
					\begin{equation}\label{mainBVP}
						\left\{\begin{array}{ll}
								F^j(x',u,\nabla u,\nabla^2u)=0, & \text{in}\,\, \Om \\
								u=f, & \text{on}\,\, \partial\Om,
							\end{array} \right.
					\end{equation}
			$j=1,2$, and assume that
			\begin{equation*}
				\Lambda_{c_1}(f)=\Lambda_{c_2}(f)
			\end{equation*}
			for all $f\in U_{\delta}:=\{h\in C^s(\partial\Om):||h||_{C^s(\partial\Om)}<\delta\}$, where $\delta>0$ is sufficiently small.
			\item Assume either that 
			\begin{enumerate}
				\item $n=3$, $\Gamma\subset\partial\Om$ be open and $\Gamma\neq\emptyset$ or 
				\item $n>3$, $\Om\subset \{x_{n-1}>0\}$, $\Gamma\subset\partial\Om$ be open, $\Gamma\neq\emptyset$ and that $\partial\Om\setminus\Gamma\subset\{x_{n-1}=0\}$ or
				\item $n>3$, $\Om$ is a strict subset of some ball $B\subset\R^{n-1}$, $\Gamma\subset\partial\Om$ be open, $\Gamma\neq\emptyset$ and that $\partial\Om\setminus\Gamma\subset\partial B$.
			\end{enumerate}
			In addition assume that
			\begin{equation*}
				\Lambda^{\Gamma}_{c_1}(f)=\Lambda^{\Gamma}_{c_2}(f)
			\end{equation*}
			for all $f\in U_{\delta}$, $\spt(f)\subset\Gamma$, where $\delta>0$ is sufficiently small and $\Lambda^{\Gamma}_{c_j}$ are the partial DN maps associated to \eqref{mainBVP} for $j=1,2$.
		\end{enumerate}
		Then in the cases $(1)$ and $(2)$ we have for $\lambda\neq0$ \[\partial_{x_n}^mc_1(x',0)=\lambda\partial_{x_n}^mc_2(x',0),\quad \text{in } \Om,\,m\geq0.\]
	\end{thm}
	The assumption $\partial_{x_n}c_j(x',0)=0$ is needed in order for $u\equiv0$ to be a solution to \eqref{mainBVP}, and this is used to prove the well-posedness. The condition $\partial_{x_n}^2c_j(x',0)=0$ is assumed in order for the method to work and it is not known if it could be removed.
	
	As an immediate corollary of Theorem \ref{main thm} we get the following.
	
	\begin{cor}
		Assume the conditions in Theorem \ref{main thm} and assume additionally that $c_j$ are real analytic with respect to $x_n$.
		Then for $\lambda\neq0$ we have
		\[c_1(x)=\lambda c_2(x),\quad x\in\Om\times\R.\]
	\end{cor}

	In section \ref{section proof of main} we give a full proof of Theorem \ref{main thm} and as mentioned, it will use higher order linearization together with complex geometric optics (CGO) solutions. In the proof we first linearize \eqref{mainBVP} at $u\equiv0$ and the DN map at $f=0$. We see that the linearization of \eqref{mainBVP} correspond to a conductivity equation where the conductivity is $c_j(x',0)$. The first linearization of the DN map maps a boundary value $f$ to $\partial_{\nu}v|_{\partial\Om}$ where $v$ is a solution to the conductivity equation. We will show that $c_j(x',0)$ can be recovered up to a multiplicative constant with the knowledge of this (partial) DN map with the help of boundary determination for a first order perturbation of the Laplacian from \cite{Brown2006} (\cite{Imanuvilov2010} for $n=3$ and \cite{MR873380}, \cite{Isakov2007} for $n>3$). In the full data case the higher order linearizations lead to an integral equality
	\begin{equation*}
		\int_{\Om}\left(\partial_{x_n}^{m+4}c_1(x',0)-\lambda\partial_{x_n}^{m+4}c_2(x',0)\right)\prod_{N=1}^{m+3}v^{l_N}\,dx'=0.
	\end{equation*}
	where $v^{l_N}$ are solutions to the first linearization. For the partial data cases, we need a special solution $v^{(0)}$ which is positive in $\Om$ and vanishes on $\partial\Om\setminus\Gamma$. With the help of this function we get the integral identity
	\begin{equation*}
		\int_{\Om}\left(\partial_{x_n}^{m+4}c_1(x',0)-\lambda\partial_{x_n}^{m+4}c_2(x',0)\right)v^{(0)}\prod_{N=1}^{m+3}v^{l_N}\,dx'=0.
	\end{equation*}
	Again $v^{l_N}$ are solutions to the first linearization. In both cases, choosing two of $v^{l_N}$ to be real or imaginary parts of CGO solutions and the rest equal to $1$ we get that $\partial_{x_n}^{m+4}c_1(x',0)=\lambda\partial_{x_n}^{m+4}c_2(x',0)$ (for $n=3$ \cite{MR2387648}, for $n>3$ \cite{MR873380}).
	
	This method has received a lot of attention in various situations lately. Linearization has already been used in a parabolic case in \cite{MR1233645} where the author shows that the first linearization of the nonlinear DN map is the DN map of a linear equation. Thus one can use the theory of inverse problems for linear equations. Also nonlinear elliptic cases have been studied, for example in \cite{MR1295934}, \cite{MR1465069}. As mentioned above, the method of higher order linearization was first used in \cite{MR3802298} for a nonlinear wave equation. After that there were two simultaneously published articles (\cite{MR4188325}, \cite{MR4104456}) in which higher order linearization was introduced to nonlinear elliptic equations of the type $\Delta u +a(x,u)=0$. The important thing in this method was that it used the nonlinearity as a tool. In \cite{MR4052205}, \cite{MR4269409} the method was further developed for the case $\Delta u +a(x,u)=0$ in inverse problems with partial data. See also \cite{Salo2022} and \cite{MR4332042} for more results on the special case of a power type nonlinearity.
	
	After these, there have been several articles using this method for different nonlinear elliptic equations. Different cases of nonlinear conductivity equations have had a treatment in \cite{MR4300916}, \cite{kian2020partial}. This method has also been used in the case of a nonlinear magnetic Schrödinger equation (\cite{lai2020partial}) and in inverse transport and diffusion problems \cite{lai2021inverse}. See also \cite{MR4216606} for a semilinear elliptic equation with gradient nonlinearities and \cite{lai2020inverse} for the case of fractional semilinear elliptic equations.
	
	There are also works in inverse problems that have considered the minimal surface equation. The Euclidean case has had a treatment in \cite{Munoz2020} where the authors consider a quasilinear conductivity depending on a function $u$ and its gradient.
	
	Also while writing this article we have learned that C\u{a}t\u{a}lin I. Cârstea, Matti Lassas, Tony Liimatainen and Lauri Oksanen are working on an inverse problem involving minimal surface equation on a Riemannian manifold in their upcoming preprint \cite{Carstea2022}. They simultaneously and independently prove a result similar to Theorem \ref{main thm}. In their work it is shown that from the knowledge of the DN map of the minimal surface equation it is possible to determine a $2$-dimensional Riemannian manifold $(\Sigma,g)$. We agreed with them to publish our preprints at the same time on the same preprint server.
	
	
	This article is organized as follows. In Section \ref{section well-posedness} we prove well-posedness for a general nonlinear boundary value problem and we describe the first and second order linearizations for the general case. Section \ref{section derivation} is dedicated to the derivation of the minimal surface equation on a manifold with conformally Euclidean metric. Section \ref{section minimal surface} consists of describing the setting for Theorem \ref{main thm} and then calculating the first and second order linearizations in this setting. Finally, we will use higher order linearization to prove Theorem \ref{main thm} in Section \ref{section proof of main}.
	\\
	
	\textbf{Acknowledgements.} The author was supported by the Finnish Centre of Excellence in Inverse Modelling and Imaging (Academy of Finland grant 284715).
	The author would like to thank Mikko Salo for helpful discussions on the minimal surface equation and everything related to inverse problems. The author would also like to thank the anonymous referees for their useful comments and suggested improvements.

	%
	%

	\section{Well-posedness and linearizations}\label{section well-posedness}
	
	In this section, we consider general equations $F(x,u,\nabla u,\nabla^2u)=0$ and in later sections apply these methods. Let $\Om\subset\R^n, n\geq2$ be a bounded domain with $C^{\infty}$ boundary and let
	$F\colon \bar{\Om}\times\R\times\Rnn\times\R^{n^2}\to\R,$ be a $C^{\infty}$ function. Consider next the boundary value problem
	\begin{equation}\label{eq F}
		\left\{\begin{array}{ll}
			F(x,u,\nabla u,\nabla^2u)=0, & \text{in}\,\, \Om \\
			u=f, & \text{on}\,\, \partial\Om,
		\end{array} \right.
	\end{equation}
	where $f\in C^s(\partial\Om)$ and $\nabla u, \nabla^2u$ denote the gradient and Hessian of $u$, respectively. In addition let $F(x,0,0,0)=0$ which guarantees that $u\equiv0$ is a solution to \eqref{eq F} with $f=0$.

	Next we prove well-posedness for \eqref{eq F} using the implicit function theorem on Banach spaces (\cite[Theorem 10.6 and Remark 10.5]{MR2028503}). In what follows, we denote for $m\times n$ matrices $A=(a_{ij}), B=(b_{ij})$ the matrix product
	\[A:B=\sum_{i=1}^{m}\sum_{j=1}^{n}a_{ij}b_{ij}\]
	and $\nabla_PF$ is the matrix with elements $\partial_{P_{ij}}F$. Also a linear differential operator $Lu=A(x):\nabla^2u+b(x)\cdot\nabla u+c(x)u$ is strictly elliptic (\cite{GilbargTrudinger}) in $\Om$ if for some constants $\lambda,\Lambda>0$ we have
	\[\lambda|\xi|^2\leq \xi^{T}A\xi\leq\Lambda|\xi|^2,\quad x\in\bar{\Om},\]
	for all $\xi\in\Rnn\setminus\{0\}$. Here $A$ is a symmetric $n\times n$ matrix.
	
	\begin{prop}\label{well-pos prop}
		Let $F\colon\Om\times\R\times\Rnn\times\R^{n^2}\to\R$ be a $C^{\infty}$ mapping with $F(x,0,0,0)=0$. 
		Furthermore assume that the map
		\begin{equation*}
			v\mapsto L(v):=\partial_uF(x,0,0,0)v+\nabla F(x,0,0,0)\cdot\nabla v+\nabla_PF(x,0,0,0):\nabla^2v
		\end{equation*}
		is injective on $H^1_0(\Om)$ and that the operator $L$ is strictly elliptic. Let $s>3, s\notin\N$. Then there exists $C, \delta>0$ such that for any 
		\[f\in U_{\delta}:=\{h\in C^s(\partial\Om):||h||_{C^s(\partial\Om)}<\delta\}\]
		the boundary value problem
		\begin{equation*}
			\left\{\begin{array}{ll}
				F(x,u,\nabla u,\nabla^2u)=0, & \text{in}\,\, \Om \\
				u=f, & \text{on}\,\, \partial\Om,
			\end{array} \right.
		\end{equation*}
		has a unique small solution $u=u_f$ which satisfies
		\begin{equation*}
			||u||_{C^s(\bar{\Om})}\leq C||f||_{C^s(\partial\Om)}.
		\end{equation*}
		Moreover the following mappings are $C^{\infty}$ maps
		\begin{align*}
			S&\colon U_{\delta}\to C^s(\bar{\Om}), &&f\mapsto u_f,\\
			\Lambda&\colon U_{\delta}\to C^{s-1}(\partial\Om), &&f\mapsto \partial_{\nu}u_f|_{\partial\Om}.
		\end{align*}
	\end{prop}
	
	\begin{proof}
		Let $X=C^s(\partial\Om), Y=C^s(\bar{\Om}), Z=C^{s-2}(\bar{\Om})\times C^s(\partial\Om)$ and 
		\begin{equation*}
			T\colon X\times Y\to Z,\quad T(f,u)=\left(F(x,u,\nabla u,\nabla^2u),u|_{\partial\Om}-f\right)
		\end{equation*}
		Since $u|_{\partial\Om}, f\in C^s(\partial\Om)$, $u\in C^s(\bar{\Om})$ and $F\in C^{\infty}$, the map $T$ really has this mapping property.
		
		Next we show that the map $u\mapsto F(x,u,\nabla u,\nabla^2u)$ is a $C^{\infty}$ map $C^s(\bar{\Om})\to C^{s-2}(\bar{\Om})$. This is done by using a Taylor expansion. Write $\lambda=(z,p,P)\in\R\times\Rnn\times\R^{n^2}$ and expand $F(x,\,\cdot\,)$ at $\mu\in\R\times\Rnn\times\R^{n^2}$:
		\begin{equation*}
			F(x,\lambda+\mu)=\sum_{|\alpha|\leq k}\frac{D^{\alpha}_{\mu}F(x,\lambda)}{\alpha!}\mu^{\alpha}+\sum_{|\beta|=k+1}R_{\beta}(\lambda+\mu)\mu^{\beta},
		\end{equation*}
		where 
		\begin{equation*}
			R_{\beta}(\lambda+\mu)=\frac{|\beta|}{\beta!}\int_{0}^{1}(1-t)^{|\beta|-1}D_{\mu}^{\beta}F(x,\lambda+t\mu)\,dt.
		\end{equation*}
		Now let $u\in C^s(\bar{\Om})$ be fixed, $\lambda=(u,\nabla u,\nabla^2u)$ and let $\mu=(h,\nabla h,\nabla^2h), h\in C^{s}(\bar{\Om})$ be such that $||\mu||_{C^{1,\alpha}(\bar{\Om})}\leq 1.$
		It is enough to show that the map $u\mapsto D^{\alpha}_{\mu}F(x,u,\nabla u,\nabla^2u)$ is continuous for all $\alpha$ and
		\begin{equation*}
			R_{\beta}(\lambda+\mu)=o(\mu^k) \quad\text{in}\quad C^s(\bar{\Om}).
		\end{equation*}
		
		Firstly, since the composition of a $C^{\infty}$ function $F$ with a $C^{s-2}$ function is again a $C^{s-2}$ function (\cite[Theorem A.8]{MR602181}), we have the continuity. The space $C^s(\bar{\Om})$ is an algebra under pointwise multiplication (\cite[Theorem A.7]{MR602181}), and thus
		\begin{align*}
			||R_{\beta}(\lambda+\mu)\mu^{\beta}||_{C^s(\bar{\Om})}&\leq C\left(||R_{\beta}(\lambda+\mu)||_{C(\bar{\Om})}||\mu^{\beta}||_{C^s(\bar{\Om})}+||R_{\beta}(\lambda+\mu)||_{C^s(\bar{\Om})}||\mu^{\beta}||_{C(\bar{\Om})}\right)\\
			&\leq C||R_{\beta}(\lambda+\mu)||_{C^s(\bar{\Om})}||\mu||_{C^s(\bar{\Om})}^{|\beta|}\\
			&\leq C ||\mu||_{C^s(\bar{\Om})}^{|\beta|} \frac{|\beta|}{\beta!}\int_{0}^{1}(1-t)^{|\beta|-1}||D_{\mu}^{\beta}F(x,\lambda+t\mu)||_{C^s(\bar{\Om})}\,dt\\
			&\leq C_{F,u} ||\mu||_{C^s(\bar{\Om})}^{|\beta|} \frac{|\beta|}{\beta!}\int_{0}^{1}(1-t)^{|\beta|-1}\,dt
		\end{align*}
		where $||D_{\mu}^{\beta}F(x,\lambda+t\mu)||_{C^s(\bar{\Om})}$ is uniformly bounded in $t\in (0,1)$ and the bounding constant may depend on $u$ and $F$. This is due to $F$ being a $C^{\infty}$ function and that $u\in C^s(\bar{\Om})$.
		Now the remainder satisifes
		\begin{equation*}
			\Big|\Big|\sum_{|\beta|=k+1}R_{\beta}(\lambda+\mu)\mu^{\beta}\Big|\Big|_{C^s(\bar{\Om})}\leq C||(h,\nabla h,\nabla^2h)||_{C^s(\bar{\Om})}^{k+1}
		\end{equation*}
		and hence the map $u\mapsto F(x,u,\nabla u,\nabla^2u)$ is a $C^{\infty}$ map $C^s(\bar{\Om})\to\R$.
		
		By the assumption $F(x,0,0,0)=0$ we have $T(0,0)=0$. Also $D_uT(0,0)$ is linear and
		\[D_uT(0,0)v=(\partial_uF(x,0,0,0)v+\nabla_pF(x,0,0,0)\cdot \nabla v+\nabla_PF(x,0,0,0): \nabla^2v,v|_{\partial\Om}).\]
		The mapping $v\mapsto L(v)$ is injective and $v\equiv0$ is a solution to 
		\begin{equation}\label{lineqF}
			\left\{\begin{array}{ll}
				\partial_uF(x,0,0,0)v+\nabla_pF(x,0,0,0)\cdot \nabla v+\nabla_PF(x,0,0,0): \nabla^2v=H, & \text{in}\,\, \Om \\
				v=g, & \text{on}\,\, \partial\Om,
			\end{array} \right.
		\end{equation}
		when $H=g=0$. Using Fredholm alternative (\cite[Theorem 6.15]{GilbargTrudinger}) the boundary value problem \eqref{lineqF} has a unique solution for all $H$ and $g$. Thus $D_uT(0,0)$ is surjective.
		
		Then by the implicit function theorem there exists $\delta>0$ and $U_{\delta}:=B(0,\delta)\subset X=C^s(\partial\Om)$ and a $C^{\infty}$ map $S\colon U_{\delta}\to Y=C^s(\bar{\Om})$ such that $T(f,S(f))=0$. Also, for small enough $f\in U_{\delta}$ (not necessarily the same $\delta$) and $u_f\in C^{s}(\bar{\Om})$, $S(f)=u_f$ is the only solution of $T(f,u_f)=0$. Moreover, since $S$ is Lipschitz continuous and $S(0)=0$, for $u=S(f)$ we have
		\begin{equation*}
			||u||_{C^s(\bar{\Om})}\leq C||f||_{C^s(\partial\Om)}.
		\end{equation*}
		Also the mapping $\Lambda$ is a well defined $C^{\infty}$ map between $U_{\delta}$ and $C^{s-1}(\partial\Om)$ since taking a normal derivative is a linear map from $C^s(\Om)$ to $C^{s-1}(\partial\Om)$.
	\end{proof}
	
	In order to use the method of higher order linearization, we calculate formally the first and second order linearizations of \eqref{eq F} and the corresponding DN map. This formal looking calculation can be justified as in \cite{MR4188325}. 
	
	Let us begin by assuming that for
	
	\begin{equation}\label{eqFj}
		\left\{\begin{array}{ll}
			F^{j}(x,u,\nabla u,\nabla^2u)=0, & \text{in}\,\, \Om \\
			u=f, & \text{on}\,\, \partial\Om,
		\end{array} \right.
	\end{equation}
	$j=1,2$, we have $\Lambda_{F^{1}}(f)=\Lambda_{F^{2}}(f)$ for all $f\in C^s(\partial\Om)$ with $||f||_{C^s(\partial\Om)}\leq \delta,$ for $\delta>0$ sufficiently small.
	In order to find the linearizations, let $\e_1,\ldots,\e_k$ be sufficiently small numbers and $f_1,\ldots,f_k\in C^s(\partial\Om).$ Let $u_j(x,\e_1,\ldots,\e_k)$ be the unique small solution to
	
	\begin{equation}\label{eqFjeps}
		\left\{\begin{array}{ll}
			F^{j}(x,u_j,\nabla u_j,\nabla^2u_j)=0, & \text{in}\,\, \Om \\
			u_j=\sum_{m=1}^{k}\e_mf_m, & \text{on}\,\, \partial\Om,
		\end{array} \right.
	\end{equation}
	for $j=1,2$. Differentiate this with respect to $\e_l$, $l\in\{1,\ldots,k\}$, and evaluate at $\e_1=\ldots=\e_k=0$ to get
	\begin{equation}\label{lineqFjeps}
		\left\{\begin{array}{ll}
			\partial_uF^{j}(x,0,0,0)v_j^l+\nabla_pF^{j}(x,0,0,0)\cdot \nabla v_j^l+\nabla_PF^{j}(x,0,0,0):\nabla^2v_j^l=0, & \text{in}\,\, \Om \\
			v_j^l=f_l, & \text{on}\,\, \partial\Om,
		\end{array} \right.
	\end{equation}
	where $v_j^l:=\partial_{\e_l}u_j(x,\e_1,\ldots,\e_k)\big|_{\e_1=\ldots=\e_k=0}$. The boundary value problem \eqref{lineqFjeps} has a unique solution if we assume that the map 
	\[v\mapsto L(v)=\partial_uF^j(x,0,0,0)v+\nabla_pF^j(x,0,0,0)\cdot \nabla v+\nabla_PF^j(x,0,0,0): \nabla^2v\] 
	is injective on $H^1_0(\Om)$ and assume strict ellipticity of the operator $L$. At this point, we would like to see what exactly is the first linearization and see if some information can be recovered about the coefficents $\partial_uF^j(x,0,0,0)$, $\nabla_pF^j(x,0,0,0)$, $\nabla_PF^j(x,0,0,0)$ from the knowledge of the DN maps corresponding to \eqref{eqFj} for $j=1,2$. What actually can be recovered depends on the equation at hand.
	
	Let us next differentiate \eqref{eqFjeps} first with respect to $\e_l$ and then with respect to $\e_a, a\neq l$:
	\begin{align*}
		I_j:=&\:\partial_{\e_a\e_l}^2F^{j}(x,u_j,\nabla u_j,\nabla^2u_j)\\
		=&\:\partial_{\e_a}\left( \partial_uF^{j}(x,u_j,\nabla u_j,\nabla^2u_j)\partial_{\e_l}u_j\right)\\
		+&\:\partial_{\e_a}\left(\sumn\partial_{p_i}F^{j}(x,u_j,\nabla u_j,\nabla^2u_j)\partial_{x_i}\partial_{\e_l}u_j\right)\\
		+&\:\partial_{\e_a}\left(\sumjk\partial_{P_{jk}}F^{j}(x,u_j,\nabla u_j,\nabla^2u_j)\partial_{x_jx_k}^2\partial_{\e_l}u_j\right)\\
		:=&\:I_{j,1}+I_{j,2}+I_{j,3}.
	\end{align*}
	Then we expand these one by one:
	\begin{align*}
		I_{j,1}&=\partial_uF^{j}\partial_{\e_a\e_l}^2u_j+\partial_u^2F^{j}\partial_{\e_a}u_j\partial_{\e_l}u_j+\sumn\partial_{p_i}\partial_uF^{j}\partial_{x_i}\partial_{\e_a}u_j\partial_{\e_l}u_j\\
		&+\sumjk\partial_{P_{jk}}\partial_uF^{j}\partial_{x_jx_k}^2\partial_{\e_a}u_j\partial_{\e_l}u_j,
	\end{align*}
	
	\begin{align*}
		I_{j,2}&=\sumn\Bigg(\partial_{p_i}F^{j}\partial_{x_i}\partial_{\e_a\e_l}^2u_j+\partial_u\partial_{p_i}F^{j}\partial_{\e_a}u_j\partial_{x_i}\partial_{\e_l}u_j\\
		&+\sum_{r=1}^{n}\partial_{p_r}\partial_{p_i}F^{j}\partial_{x_r}\partial_{\e_a}u_j\partial_{x_i}\partial_{\e_l}u_j+\sumjk\partial_{P_{jk}}\partial_{p_i}F^{j}\partial_{x_jx_k}^2\partial_{\e_a}u_j\partial_{x_i}\partial_{\e_l}u_j\Bigg),
	\end{align*}
	
	\begin{align*}
		I_{j,3}&=\sumjk\Bigg( \partial_{P_{jk}}F^j\partial_{x_jx_k}^2\partial_{\e_a\e_l}^2u_j+\partial_u\partial_{P_{jk}}F\partial_{\e_a}u_j\partial_{x_jx_k}^2\partial_{\e_l}u_j\\
		&+\sumn\partial_{p_i}\partial_{P_{jk}}F^j\partial_{x_i}\partial_{\e_a}u_j\partial_{x_jx_k}^2\partial_{\e_l}u_j+\sum_{r,t=1}^{n}\partial_{P_{rt}}\partial_{P_{jk}}F^j\partial_{x_rx_t}^2\partial_{\e_a}u_j\partial_{x_jx_k}^2\partial_{\e_l}u_j\Bigg).
	\end{align*}
	Evaluate $I_j$ at $\e_1=\ldots=\e_k=0$ and denote $w_j^{(al)}=(\partial_{\e_a\e_l}^2u_j)(x,\e_1,\ldots,\e_k)|_{\e_1=\ldots=\e_k=0}$ to have
	\begin{align}\label{secondlinFjeps}
		I_j&=\partial_uF^j(x,0,0,0)w_j^{(al)}+\partial_u^2F^j(x,0,0,0)v^lv^a\\\notag
		&+\left(\big(\nabla_p(\partial_uF^j)\big)(x,0,0,0)\cdot \nabla v^a+\big(\nabla_P(\partial_uF^j)\big)(x,0,0,0):\nabla^2v^a\right)v^l\\\notag
		&+\nabla_pF^j(x,0,0,0)\cdot \nabla w_j^{(al)}+\big(\nabla_p(\partial_uF^j)\big)(x,0,0,0)\cdot \nabla v^lv^a\\\notag
		&+\sumn\left(\big(\nabla_p(\partial_{p_i}F^j)\big)(x,0,0,0)\cdot \nabla v^a+\big(\nabla_P(\partial_{p_i}F^j)\big(x,0,0,0):\nabla^2v^a\right)\partial_{x_i}v^l\\\notag
		&+\nabla_PF^j(x,0,0,0): \nabla^2w_j^{(al)}+\big(\nabla_P(\partial_uF^j)\big)(x,0,0,0): \nabla^2v^lv^a\\\notag
		&+\sumjk\left(\big(\nabla_p(\partial_{P_{jk}}F^j)\big)\cdot \nabla v^a+\big(\nabla_P(\partial_{P_{jk}}F^j)\big)(x,0,0,0):\nabla^2v^a\right)\partial_{x_ix_j}^2v^l
	\end{align}
	Thus $w_j^{(al)}$ satisfies the boundary value problem
	\begin{equation}\label{secondlineqFjeps}
		\left\{\begin{array}{ll}
			I_j=0, & \text{in}\,\, \Om \\
			w_j^{(al)}=0, & \text{on}\,\, \partial\Om.
		\end{array} \right.
	\end{equation}
	Next we would like to integrate $I_1-I_2$ against a solution to the adjoint of
	\begin{equation*}
		\partial_uF^{j}(x,0,0,0)v_j^l+\nabla_pF^{j}(x,0,0,0)\cdot \nabla v_j^l+\nabla_PF^{j}(x,0,0,0):\nabla^2v_j^l=0
	\end{equation*}
	and use the assumption that the DN maps associated to \eqref{eqFj} coincide for $j=1,2$ together with a completeness result to recover information about the coefficients of $I_1$ and $I_2$. Again the information that can be recovered depends on the equation and below this method is applied in the case of the minimal surface equation.
	
	What we would do next is to use an induction argument to show that from higher order linearizations it is possible to recover more information. This too will be specified below.

	\section{Mimimal surface equation on a Riemannian manifold}\label{section derivation}
	
	In this section we derive the equation \eqref{minsurfeq}. Let $(M,g),$ $M=\Rnn$, $n\geq3$, be a Riemannian manifold with the metric
	\begin{equation}\label{metricg}
		g_{ij}(x',x_n)=c(x',x_n)\delta_{ij},
	\end{equation}
	where
	\begin{equation*}\label{metricc}
		(x',x_n)\in\R^{n-1}\times\R,\; c\in C^{\infty}(\R^n),\; c(x)>0\quad\text{for all}\quad x\in\R^n.
	\end{equation*}
	These assumptions are valid for the rest of the article, unless otherwise stated.
	
	Let $u\colon \Om\subset\R^{n-1}\to\R,$ $u\in C^2(\bar{\Om})$, and consider the graph of the function $u$
	\begin{equation*}
		\text{Graph}_u=\{(x',u(x'))\colon x'\in\Om\}\subset M.
	\end{equation*}
	This graph is a minimal surface if and only if its mean curvature $H$ is equal to zero at all points on the graph. By defining 
	\begin{equation*}
		f\colon\Om\times\R\to\R,\quad f(x',x_n)=x_n-u(x'),
	\end{equation*}
	the graph of $u$ is the surface 
	\begin{equation*}
		\Sigma:=\{(x',x_n)\in\Om\times\R : f(x',x_n)=0\}.
	\end{equation*}
	The mean curvature of $\Sigma$ at $x\in\Sigma$ is the sum of principal curvatures. We omit the normalizing factor $\frac{1}{n-1}$ when calculating the mean curvature. In order to calculate the principal curvatures, we introduce the Riemannian gradient and Hessian of a function $f\colon M\to\R$:
	\begin{align*}
		\nabla_gf=g^{ij}\partial_{x_i}f\partial_{x_j},\quad \nabla_g^2f=\left(\partial^2_{x_ix_j}f-\chr{m}{ij}\partial_{x_m}f\right)_{i,j=1}^n,
	\end{align*}
	where $g^{ij}$ is the inverse of $g_{ij}$ and $\chr{m}{ij}=\frac{1}{2}g^{ml}(\partial_{x_i}g_{jl}+\partial_{x_j}g_{il}-\partial_{x_l}g_{ij})$ is the Christoffel symbol related to the metric $g$.
	Define also the Laplace-Beltrami operator, which is a trace of the Hessian (this is one way of defining it), and the norm of the gradient:
	\begin{equation*}
		\Delta_gf=\Tr(\nabla_g^2f)=g^{ij}\left(\partial^2_{x_ix_j}f-\chr{m}{ij}\partial_{x_m}f\right),\quad |\nabla_gf|_g^2=g^{ij}\partial_{x_i}f\partial_{x_j}f .
	\end{equation*}
	Now the principal curvatures of $\Sigma$ at $x\in\Sigma$ are the eigenvalues of $\nabla_g^2f(x)$ restricted to the tangent space $T_x\Sigma$ at $x$. Since $\frac{\nabla_g f(x)}{|\nabla_g f(x)|_g}$ is a normal to $\Sigma$ at the point $x$, we have $T_x\Sigma=\{\nabla_g f(x)\}^{\perp}$, or in other words, the tangent space $T_x\Sigma$ is the orthogonal complement of the vector $\nabla_g f(x)$. 
	
	Let $\{E_1,\ldots,E_{n-1}\}$ be an $g$-orthonormal basis of $T_x\Sigma$. Then $\left\{E_1,\ldots,E_{n-1},\frac{\nabla_g f(x)}{|\nabla_g f(x)|_g}\right\}$ is an orthonormal basis of $\Rnn$. Now the mean curvature of $\Sigma$ at $x\in\Sigma$ is the trace of $\nabla_g^2f(x)|_{\{\nabla_g f(x)\}^{\perp}}$:
	\begin{align*}
		H(x)&=\sum_{i=1}^{n-1}\langle\nabla_g^2f(x)E_i,E_i\rangle\\
		&=\sum_{i=1}^{n-1}\langle\nabla_g^2f(x)E_i,E_i\rangle+\left(\nabla_g^2f(x)\right)\left(\frac{\nabla_g f(x)}{|\nabla_g f(x)|_g},\frac{\nabla_g f(x)}{|\nabla_g f(x)|_g}\right)\\
		&-\left(\nabla_g^2f(x)\right)\left(\frac{\nabla_g f(x)}{|\nabla_g f(x)|_g},\frac{\nabla_g f(x)}{|\nabla_g f(x)|_g}\right)\\
		&=\Tr(\nabla_g^2f(x))-|\nabla_gf(x)|_g^{-2}\left(\nabla_g^2f(x)\right)\left(\nabla_g f(x),\nabla_g f(x)\right)\\
		&=\Delta_gf(x)-|\nabla_gf(x)|_g^{-2}\left(\nabla_g^2f(x)\right)\left(\nabla_g f(x),\nabla_g f(x)\right).
	\end{align*}
	Thus $\text{Graph}_u$ is a minimal surface if and only if 
	\begin{equation}\label{general minsurfeq}
		|\nabla_gf(x)|_g^2\Delta_gf(x)-\left(\nabla_g^2f(x)\right)\left(\nabla_g f(x),\nabla_g f(x)\right)=0\quad\text{for all}\va x\in\text{Graph}_u.
	\end{equation}
	
	Next we will calculate the minimal surface equation more explicitly using the conformally Euclidean metric \eqref{metricg}. Now $g^{ij}=c^{-1}\delta_{ij}$ is the inverse matrix of \eqref{metricg} and thus
	\begin{align*}
		\nabla_gf=c^{-1}\sum_{j=1}^n\partial_{x_j}f\partial_{x_j},\quad |\nabla_gf|_g^2=c^{-1}\sum_{i=1}^n\partial_{x_i}f\partial_{x_i}f.
	\end{align*}
	Also the Christoffel symbol can be simplified by letting $\lambda=\frac{1}{2}\log c$ and hence $\partial_{x_i}\lambda=\frac{1}{2}c^{-1}\partial_{x_i}c$. Then
	\begin{align*}
		\chr{m}{ij}&=\frac{1}{2}g^{ml}\left(\partial_{x_i}g_{jl}+\partial_{x_j}g_{il}-\partial_{x_l}g_{ij}\right)\\
		&=\frac{1}{2}c^{-1}\left(\partial_{x_i}c\delta_{jm}+\partial_{x_j}c\delta_{im}-\partial_{x_m}c\delta_{ij}\right)\\
		&=\partial_{x_i}\lambda\delta_{jm}+\partial_{x_j}\lambda\delta_{im}-\partial_{x_m}\lambda\delta_{ij}.
	\end{align*}
	Let us next calculate the two parts of \eqref{general minsurfeq} separately, starting from
	\begin{align*}
		c^2|\nabla_gf|_g^2\Delta_gf
		&=\left(\sum_{i=1}^n\partial_{x_i}f\partial_{x_i}f\right)\left(\sum_{i=1}^n\partial^2_{x_ix_i}f-\sum_{m=1}^n\left(\sum_{i,j=1}^{n}\delta_{ij}\chr{m}{ij}\right)\partial_{x_m}f\right)\\
		&=\left(\sum_{i=1}^n\partial_{x_i}f\partial_{x_i}f\right)\left(\sum_{i=1}^n\partial^2_{x_ix_i}f-(2-n)\sum_{m=1}^n\partial_{x_m}\lambda\partial_{x_m}f\right),
	\end{align*}
	and the other part becomes
	\begin{align*}
		&c^2\nabla_g^2f(\nabla_gf,\nabla_gf)\\
		&=\left(\partial_{x_ix_j}^2f-\chr{m}{ij}\partial_{x_m}f\right) cg^{ia}\partial_{x_a}fcg^{jb}\partial_{x_b}f\\
		&=\sum_{i,j=1}^n\left(\partial_{x_ix_j}^2f-\sum_{m=1}^n\left(\partial_{x_i}\lambda\delta_{jm}+\partial_{x_j}\lambda\delta_{im}-\partial_{x_m}\lambda\delta_{ij}\right)\partial_{x_m}f\right)\partial_{x_i}f\partial_{x_j}f\\
		&=\sum_{i,j=1}^n\left(\partial_{x_ix_j}^2f-\partial_{x_i}\lambda\partial_{x_j}f-\partial_{x_j}\lambda\partial_{x_i}f+\left(\sum_{m=1}^n\partial_{x_m}\lambda\partial_{x_m}f\right)\delta_{ij}\right)\partial_{x_i}f\partial_{x_j}f\\
		&=\sum_{i,j=1}^n\left(\partial_{x_ix_j}^2f-2\partial_{x_i}\lambda\partial_{x_j}f\right)\partial_{x_i}f\partial_{x_j}f+\left(\sum_{i=1}^n\partial_{x_i}f\partial_{x_i}f\right)\left(\sum_{m=1}^n\partial_{x_m}\lambda\partial_{x_m}f\right).
	\end{align*}
	Now
	\begin{align*}
		&c^2|\nabla_gf|_g^2\Delta_gf-c^2\nabla_g^2f\left(\nabla_g f,\nabla_g f\right)\\
		&=\left(\sum_{i=1}^n\left(\partial_{x_i}f\right)^2\right)\left(\sum_{i=1}^n\partial^2_{x_ix_i}f+(n-3)\sum_{m=1}^n\partial_{x_m}\lambda\partial_{x_m}f\right)\\
		&-\sum_{i,j=1}^n\left(\partial_{x_ix_j}^2f-2\partial_{x_i}\lambda\partial_{x_j}f\right)\partial_{x_i}f\partial_{x_j}f.
	\end{align*}
	Plugging the above to \eqref{general minsurfeq}, we get that $\Sigma$ is a minimal surface if and only if
	\begin{align}\label{minsurfeqf}
		&\left(\sum_{i=1}^n\left(\partial_{x_i}f\right)^2\right)\left(\sum_{i=1}^n\partial^2_{x_ix_i}f+(n-3)\sum_{m=1}^n\partial_{x_m}\lambda\partial_{x_m}f\right)\\
		&-\sum_{i,j=1}^n\left(\partial_{x_ix_j}^2f-2\partial_{x_i}\lambda\partial_{x_j}f\right)\partial_{x_i}f\partial_{x_j}f=0\notag
	\end{align}
	for $x\in\Sigma$.
	
	Insert next $f(x',x_n)=x_n-u(x')$ to the above in order to get an equation in terms of the function $u$. Then the first line of \eqref{minsurfeqf} becomes (note that $\partial_{x_n}u=0$)
	\begin{align*}
		&\left(\sum_{i=1}^n\delta_{in}-2\delta_{in}\partial_{x_i}u+\left(\partial_{x_i}u\right)^2\right)\left(-\sum_{i=1}^n\partial^2_{x_ix_i}u+(n-3)\sum_{m=1}^n\partial_{x_m}\lambda(\delta_{mn}-\partial_{x_m}u)\right)\\
		&=\left(1+|\nabla u|^2\right)\left(-\Delta u+(n-3)(\partial_{x_n}\lambda-\nabla_{x'}\lambda\cdot\nabla u)\right).
	\end{align*}
	The second line is equal to
	\begin{align*}
		&-\sum_{i,j=1}^n(-\partial_{x_ix_j}^2u-2\partial_{x_i}\lambda(\delta_{jn}-\partial_{x_j}u))(\delta_{in}\delta_{jn}-\delta_{in}\partial_{x_j}u-\delta_{jn}\partial_{x_i}u+\partial_{x_i}u\partial_{x_j}u)\\
		&=\sum_{i,j=1}^n\partial_{x_ix_j}^2u\partial_{x_i}u\partial_{x_j}u+2\partial_{x_n}\lambda-2\sum_{i=1}^{n-1}\partial_{x_i}\lambda\partial_{x_i}u+2\sum_{j=1}^{n-1}\partial_{x_n}\lambda(\partial_{x_j}u)^2\\
		&-2\sum_{i,j=1}^{n-1}\partial_{x_i}\lambda\partial_{x_i}u(\partial_{x_j}u)^2\\
		&=\nabla u^T\nabla^2u\nabla u+2\partial_{x_n}\lambda-2\nabla_{x'}\lambda\cdot\nabla u+2\partial_{x_n}\lambda|\nabla u|^2-2\nabla_{x'}\lambda\cdot\nabla u|\nabla u|^2\\
		&=\nabla u^T\nabla^2u\nabla u + 2\left(\partial_{x_n}\lambda-\nabla_{x'}\lambda\cdot\nabla u\right)(1+|\nabla u|^2).
	\end{align*}
	Combining these two, we get that $\text{Graph}_u$ is a minimal surface if and only if the function $u$ satisfies the following minimal surface equation
	\begin{equation}\label{minsurfequ}
		-\Delta u+\frac{\nabla u^T\nabla^2u\nabla u}{1+|\nabla u|^2}-\frac{n-1}{2c(x',u(x'))}\left(\nabla_{x'}c(x',u(x'))\cdot\nabla u-\partial_{x_n}c(x',u(x'))\right)=0
	\end{equation}
	for all $x'\in\Om$. Multiplying both sides with $(1+|\nabla u|^2)^{-1/2}$ gives
	\begin{align*}
		&-\dive\left(\frac{\nabla u}{(1+|\nabla u|^2)^{1/2}}\right) - \sum_{i,j=1}^{n-1}\frac{\partial_{x_j}u}{(1+|\nabla u|^2)^{1/2}}\frac{\partial_{x_j}c\delta_{ii}}{2c} + \frac{(n-1)\partial_{x_n}c}{2c(1+|\nabla u|^2)^{1/2}}\\
		&=-\dive\left(\frac{\nabla u}{(1+|\nabla u|^2)^{1/2}}\right) - \sum_{i,j=1}^{n-1}\frac{\partial_{x_j}u}{(1+|\nabla u|^2)^{1/2}}\chr{i}{ij} + \frac{(n-1)\partial_{x_n}c}{2c(1+|\nabla u|^2)^{1/2}}\\
		&=-\dive_{g_{n-1}}\left(\frac{\nabla u}{(1+|\nabla u|^2)^{1/2}}\right) + \frac{(n-1)\partial_{x_n}c}{2c(1+|\nabla u|^2)^{1/2}}\\
		&=0.
	\end{align*}
	In the Euclidean setting, $c\equiv1$, this is the more familiar Euclidean minimal surface equation.

	\section{Preliminaries for the higher order linearization}\label{section minimal surface}
	
	In this section we use the method of higher order linearization on the minimal surface equation derived in the previous section. From now on, assume that $\Om\subset\R^{n-1}$ is a bounded domain. Let us start by looking at the assumptions of Proposition \ref{well-pos prop}, where it is assumed that $u\equiv0$ is a solution to \eqref{minsurfequ}. This leads to the condition that
	\begin{equation}\label{c oletus}
		\partial_{x_n}c(x',0)=0,\quad x'\in\Om,
	\end{equation}
	which can be seen as follows. For a constant function $u\colon\Om\to\R, u(x')=d$ to be a solution to \eqref{minsurfequ} is equivalent with
	\begin{equation*}
		-\frac{n-1}{2c(x',d)}\partial_{x_n}c(x',d)=0\quad\text{for all}\va x'\in\Om,
	\end{equation*}
	which is equivalent with $\partial_{x_n}c(x',d)=0$ for all $x'\in\Om$. The assumption \eqref{c oletus} comes by setting $d=0$.
	
	Next we will focus on the boundary value problem \eqref{eq F} in the setting described above and calculate the first and second order linearizations of \eqref{eq F} and the corresponding linearizations of the DN map. This could be done directly from \eqref{minsurfequ} but we will follow the general method in Section \ref{section well-posedness} and begin by defining a function $F\colon \R^{n^2}\to\R$,
	\begin{align*}\label{Fminsurf}
		F(x',u,p,P):=&\:-\sum_{i=1}^{n-1}P_{ii} - \frac{n-1}{2c(x',u)}\left(\sum_{i=1}^{n-1}p_i\partial_{x_i}c(x',u)-\partial_{x_n}c(x',u)\right) \\\notag
		&+ \frac{1}{1+|p|^2}\sum_{i,j=1}^{n-1}P_{ij}p_ip_j.\notag
	\end{align*}
	Here $p=(p_1,\ldots,p_{n-1})\in\R^{n-1}, P=(P_{ij})$ is an $(n-1)\times(n-1)$ matrix and $x'\in\R^{n-1}$. Then \eqref{minsurfequ} is equivalent with $F(x',u,\nabla u,\nabla^2u)=0$ for all $x'\in\Om$.

	Let us start with the first linearization. For this, let $x'\in\Om$. As shown in Section \ref{section well-posedness}, we need to differentiate $F$ with respect to $u,p$ and $P$. The first derivatives with respect to variable $P$ are
	\begin{equation*}\label{eka deri P}
		\partial_{P_{kl}}F(x',u,p,P)=-\delta_{kl}+\frac{p_kp_l}{1+|p|^2}.
	\end{equation*}
	When evaluated at $\e_1=\ldots=\e_k=0$, we get
	\begin{equation*}\label{eka P deriv}
		\partial_{P_{kl}}F(x',0,0,0)=-\delta_{kl}.
	\end{equation*}

	Next calculation is $\nabla_pF$:
	\begin{align*}\label{eka deri p}
		\partial_{p_k}F(x',u,p,P)&=-\frac{n-1}{2c}\partial_{x_k}c-\frac{2p_k}{(1+|p|^2)^2}\sum_{i,j=1}^{n-1}P_{ij}p_ip_j\\
		&+\frac{1}{1+|p|^2}\sum_{i,j=1}^{n-1}P_{ij}(\delta_{ik}p_j+\delta_{jk}p_i)\\
		&=-\frac{n-1}{2c}\partial_{x_k}c-\frac{2p_k}{(1+|p|^2)^2}\sum_{i,j=1}^{n-1}P_{ij}p_ip_j\\
		&+\frac{1}{1+|p|^2}\left(\sum_{j=1}^{n-1}P_{kj}p_j+\sum_{i=1}^{n-1}P_{ik}p_i\right).
	\end{align*}
	Setting $\e_1=\ldots=\e_k=0$, this becomes
	\begin{equation*}\label{eka p deriv}
		\partial_{p_k}F(x',0,0,0)=-\frac{n-1}{2c(x',0)}\partial_{x_k}c(x',0).
	\end{equation*}
	Since the solution operator $S$ from Proposition \ref{well-pos prop} is smooth, the solution $u(x',\e)$ depends smoothly on $\e$ and thus $u(x',\e)\big|_{\e_1=\ldots=\e_k=0}=0$. Hence the coordinate $x_n$ is $0$ since we are on the graph of $u$.
	
	What is left to calculate is the derivative $\partial_uF$:
	\begin{align*}
		\partial_uF(x',u,p,P)&=-\frac{n-1}{2c^2}(\partial_{x_n}c)^2+\frac{n-1}{2c}\partial_{x_n}^2c\\
		&+\frac{n-1}{2c^2}\partial_{x_n}c\sum_{i=1}^{n-1}\partial_{x_i}cp_i-\frac{n-1}{2c}\sum_{i=1}^{n-1}\partial_{x_i}\partial_{x_n}cp_i.
	\end{align*} 
	Hence, when evaluated at $\e_1=\ldots=\e_k=0$
	\begin{equation*}
		\partial_uF(x',0,0,0)=\frac{n-1}{2c(x',0)}\partial_{x_n}^2c(x',0),
	\end{equation*}
	since $\partial_{x_n}c(x',0)=0$.
	
	In Theorem \ref{main thm} the condition $\partial_{x_n}^2c(x',0)=0$ is assumed and thus
	\begin{align}
		\partial_uF(x',0,0,0)&=0.
	\end{align}
	Now the first linearization \eqref{lineqFjeps} is the following boundary value problem
	\begin{equation}\label{firstlinminimal 1}
		\left\{\begin{array}{lll}
			\Delta v^l+\dfrac{n-1}{2c(x',0)}\nabla_{x'}c(x',0)\cdot \nabla v^l=0, & \text{in}\,\, \Om \\
			v^l=f_l, & \text{on}\,\, \partial\Om. \phantom{\Big|}
		\end{array} \right.
	\end{equation}
	By multiplying the first equation in \eqref{firstlinminimal 1} with $c(x',0)^{\frac{n-1}{2}}$ we see that \eqref{firstlinminimal 1} is equivalent with
	\begin{equation*}\label{firstlinminimal 2}
		\left\{\begin{array}{ll}
			\dive\left(c(x',0)^{\frac{n-1}{2}} \nabla v^l\right)=0, & \text{in}\,\, \Om \\
			v^l=f_l, & \text{on}\,\, \partial\Om.
		\end{array} \right.
	\end{equation*}
	Hence the first linearization of the DN-map \eqref{DN minimal} at $f=0$ is 
	\begin{equation}\label{DNminimal firstlin}
		(D\Lambda_c)_0\colon C^s(\partial\Om)\to C^{s-1}(\partial\Om),\quad f\mapsto \partial_{\nu}v_f\big|_{\partial\Om}.
	\end{equation}
	In dimension $3-1=2$ one can recover $c(x',0)$ up to a multiplicative constant using a boundary determination result (\cite{Brown2006}) together eith the knowledge of the partial DN-map (\cite{Imanuvilov2010}). When $n>3$, $c(x',0)$ can be recovered, again up to a multiplicative constant, combining the same boundary determination result and the DN map (\cite{MR873380}) or the partial DN map (when $\Om$ is as described in Theorem \ref{main thm} part $(2)$ \cite{Isakov2007}). Details will be shown in the proof of Theorem \ref{main thm}.

	For the second linearization, as can be seen from \eqref{secondlineqFjeps}, second derivatives of the map $F$ need to be calculated. Firstly
	\begin{align*}
		\partial_{P_{rs}}\partial_{P_{kl}}F(x',u,p,P)&=0, \\
		\partial_{P_{kl}}\partial_uF(x',u,p,P)&=\partial_u\partial_{P_{kl}}F(x',u,p,P)=0
	\end{align*}
	and hence, when evaluated at $\e_1=\ldots=\e_k=0$, these vanish. Let us next calculate other mixed derivatives. Now
	\begin{equation*}
		\partial_{p_s}\partial_{P_{kl}}F(x',u,p,P)=\frac{-2p_{s}}{(1+|p|^2)^2}p_kp_l +\frac{1}{1+|p|^2}(\delta_{ks}p_l+\delta_{ls}p_k)
	\end{equation*}
	and when evaluated at $\e_1=\ldots=\e_k=0$
	\begin{equation*}
		\partial_{p_s}\partial_{P_{kl}}F(x',0,0,0)=\partial_{P_{kl}}\partial_{p_s}F(x',0,0,0)=0.
	\end{equation*}
	Also
	\begin{align*}
		\partial_{p_k}\partial_uF(x',u,p,P)=\frac{n-1}{2c^2}\partial_{x_n}c\partial_{x_k}c-\frac{n-1}{2c}\partial_{x_k}\partial_{x_n}c.
	\end{align*}
	Setting $\e_1=\ldots=\e_k=0$ we have
	\begin{align*}
		\partial_{p_k}\partial_uF(x',0,0,0)=\partial_u\partial_{p_k}F(x',0,0,0)=0,
	\end{align*}
	since $\partial_{x_k}\partial_{x_n}c(x',0)=0$ for $k=1,\ldots,n-1$.
	
	What is left are the second derivatives with respect to the variables $p$ and $u$. Let us start from the variable $p$:
	\begin{align*}
		\partial_{p_s}\partial_{p_k}F(x',u,p,P)&=\frac{-2\delta_{sk}}{(1+|p|^2)^2}\sum_{i,j=1}^{n-1}P_{ij}p_ip_j\\
		&- 2p_k\left(\frac{-4p_s}{(1+|p|^2)^3}\sum_{i,j=1}^{n-1}P_{ij}p_ip_j+\frac{1}{(1+|p|^2)^2}\sum_{i,j=1}^{n-1}P_{ij}(\delta_{is}p_j+\delta_{js}p_i)\right)\\
		&-\frac{2p_s}{(1+|p|^2)^2}\left(\sum_{j=1}^{n-1}P_{kj}p_j+\sum_{i=1}^{n-1}P_{ik}p_i\right)+\frac{1}{1+|p|^2}(P_{ks}+P_{sk}).
	\end{align*}
	Hence when evaluating at $\e_1=\ldots=\e_k=0$
	\begin{align*}\label{toka deri p}
		\partial_{p_r}\partial_{p_k}F(x',0,0,0)=0.
	\end{align*}
	
	For the variable $u$ the second derivative reads
	\begin{align*}
		\partial_u\partial_uF(x',u,p,P)&=\frac{n-1}{c^3}(\partial_{x_n}c)^3-\frac{3(n-1)}{2c^2}\partial_{x_n}c\partial_{x_n}^2c\\
		&+\frac{n-1}{2c}\partial_{x_n}^3c
		-\frac{n-1}{c^2}(\partial_{x_n}c)^2\sum_{i=1}^{n-1}\partial_{x_i}cp_i\\
		&+\frac{n-1}{2c}\left(\partial_{x_n}^2c\sum_{i=1}^{n-1}\partial_{x_i}cp_i+\partial_{x_n}c\sum_{i=1}^{n-1}\partial_{x_i}\partial_{x_n}cp_i\right)\\
		&+\frac{n-1}{2c^2}\partial_{x_n}c\sum_{i=1}^{n-1}\partial_{x_i}\partial_{x_n}cp_i-\frac{n-1}{2c}\sum_{i=1}^{n-1}\partial_{x_i}\partial_{x_n}^2cp_i.
	\end{align*}
	Thus, letting at $\e_1=\ldots=\e_k=0$
	\begin{align*}
		\partial_u^2F(x',0,0,0)=\frac{n-1}{2c(x',0)}\partial_{x_n}^3c(x',0).
	\end{align*}

	Let us plug the calculated derivatives in \eqref{secondlinFjeps} to find out what is the second linearization:
	\begin{equation*}
		I=\tfrac{n-1}{2c(x',0)}\partial_{x_n}^2c(x',0)w^{(al)} + \tfrac{n-1}{2c(x',0)}\partial_{x_n}^3c(x',0)v^lv^a - \tfrac{n-1}{2c(x',0)}\nabla_{x'}c(x',0)\cdot\nabla w^{(al)}-\Delta w^{(al)}.
	\end{equation*}
	Here $v^a, v^l$ satisfy \eqref{firstlinminimal 1} with corresponding boundary values.
	Now the function $w^{(al)}=(\partial_{\e_a\e_l}^2u)(x',0,\ldots,0)$ solves
	\begin{equation}\label{secondlineqminimal1}
		\left\{\begin{array}{ll}
			\Delta w^{(al)}+\frac{n-1}{2c(x',0)}\nabla_{x'}c(x',0)\cdot\nabla w^{(al)}\\
			+\frac{1-n}{2c(x',0)}\partial_{x_n}^2c(x',0)w^{(al)}+\frac{1-n}{2c(x',0)}\partial_{x_n}^3c(x',0)v^lv^a=0, & \text{in}\,\, \Om \\
			w^{(al)}=0, & \text{on}\,\, \partial\Om. \phantom{\Big|}
		\end{array} \right.
	\end{equation}
	In Theorem \ref{main thm} there is an assumption that $\partial_{x_n}^2c(x',0)=0$ and thus the term $\frac{n-1}{2c(x',0)}\partial_{x_n}^2c(x',0)w^{(al)}$ vanishes. Now the boundary value problem \eqref{secondlineqminimal1} is equivalent with
	\begin{equation}\label{secondlineqminimal2}
		\left\{\begin{array}{ll}
			\dive \left(c(x',0)^{\frac{n-1}{2}}\nabla w^{(al)}\right)-\frac{n-1}{2}c(x',0)^{\frac{n-1}{2}-1}\partial_{x_n}^3c(x',0)v^lv^a=0, & \text{in}\,\, \Om \\
			w^{(al)}=0, & \text{on}\,\, \partial\Om. \phantom{\Big|}
		\end{array} \right.
	\end{equation}

	\section{Proof of Theorem \ref{main thm}}\label{section proof of main}
	
	Now we use the method of higher order linearization to prove our main result. This will also make use of the linearizations calculated in the previous section. Before the proof we state a proposition which says that products of solutions to the Schrödinger equation form a complete set in $L^1(\Om)$ for $n\geq2$ (when $n\geq3$ \cite{MR873380}, when $n=2$ \cite{MR2387648}, see also \cite{blaasten2011inverse} and \cite[Proposition 2.1]{MR4269409} where this result is stated in the following form).
	
	\begin{prop}\label{cgo}
		Let $\Om\subset\Rnn, n\geq2,$ be a bounded domain with $C^{\infty}$ boundary, $q_1,q_2\in C^{\infty}(\bar{\Om})$ and let $f\in L^{\infty}(\Om).$ Assume that
		\begin{equation*}
			\int_{\Om}fv_1v_2\,dx=0,
		\end{equation*}
		for all $v_j$ solving $(-\Delta + q_j)v_j=0$ in $\Om$. Then $f\equiv0$ in $\Om$.
	\end{prop}

	\begin{proof}[Proof of Theorem \ref{main thm}]
		The assumptions of Proposition \ref{well-pos prop} hold for our case and thus \eqref{mainBVP} is well-posed. 
		Assume now that we have two conformal factors $c_1, c_2$ on the manifold $M$. As in Section \ref{section well-posedness}, let $\e_1,\ldots,\e_{N+1}$ be sufficiently small numbers, $\e=(\e_1,\ldots,\e_{N+1})$, $f_1,\ldots,f_{N+1}\in C^s(\partial\Om)$ and $u_j(x,\e)$ be the unique small solution to 
		\begin{equation*}
			\left\{\begin{array}{ll}
				F^{j}(x,u_j,\nabla u_j,\nabla^2u_j)=0, & \text{in}\,\, \Om \\
				u_j=\sum_{m=1}^{N+1}\e_mf_m, & \text{on}\,\, \partial\Om,
			\end{array} \right.
		\end{equation*}
		for $j=1,2$, where $F^j$ is \eqref{isoF} with $c$ replaced by $c_j$.
		
		The proof now divides to the cases $(1)$and $(2)$ and we will prove first $(1)$. It is the most straightforward of these cases and the other cases are proven similarly with only minor modifications.
		
		\vspace{8pt}
		\textbf{Case $(1)$:} 
		Assume now that
		\begin{equation}\label{DN equality}
			\Lambda_{c_1}(f)=\Lambda_{c_2}(f)
		\end{equation}
		for all $f\in C^s(\partial\Om)$ sufficiently small. Now we have the corresponding $F^1, F^2$ and the first linearization of $\Lambda_{c_j}$ is \eqref{DNminimal firstlin}, with $c(x',0)$ replaced by $c_j(x',0)$, which corresponds to the conductivity equation
		\begin{equation}\label{firstlinminimalproof}
			\left\{\begin{array}{ll}
				\dive\left(c_j(x',0)^{\frac{n-1}{2}} \nabla v_j^l\right)=0, & \text{in}\,\, \Om \\
				v_j^l=f_l, & \text{on}\,\, \partial\Om,
			\end{array} \right.
		\end{equation}
		for $j=1,2$. Using boundary determination from \cite{Brown2006} for the case of Laplacian with a convection term we get for $x'\in\partial\Om$
		\begin{equation*}
			\frac{\nabla_{x'}c_1(x',0)}{c_1(x',0)}=\frac{\nabla_{x'}c_2(x',0)}{c_2(x',0)}\quad\iff\quad\nabla_{x'}\big(\ln(c_1(x',0))\big)=\nabla_{x'}\big(\ln(c_2(x',0))\big).
		\end{equation*}
		From this we get that $\nabla_{x'}(\ln(c_1(x',0))-\ln(c_2(x',0)))=\nabla_{x'}\left(\ln\frac{c_1(x',0)}{c_2(x',0)}\right)=0$ which then implies that $c_1(x',0)=\lambda c_2(x',0)$ for $x'\in\partial\Om$ and $\lambda\neq0$.
		(We can not use boundary determination for the conductivity equation (e.g. \cite{Kohn1984}) because the DN maps are different: here $f\mapsto \partial_{\nu}v_f\big|_{\partial\Om}$ instead of $f\mapsto c_1(x',0)\partial_{\nu}v_f\big|_{\partial\Om}$.) It is known that the knowledge of this linearized DN map combined with $c_1(x',0)|_{\partial\Om}=\lambda c_2(x',0)|_{\partial\Om}$ gives us $c_1(x',0)=\lambda c_2(x',0)$ in $\Om$ (\cite{MR873380}).
		
		By the gauge invariance of \eqref{firstlinminimal 1} (replacing $c_2(x',0)=\lambda^{-1}c_1(x',0)$) we have that $v_j^l$ solve the equation
		\begin{equation*}
			\left\{\begin{array}{ll}
				\dive\left(c_1(x',0)^{\frac{n-1}{2}} \nabla v_j^l\right)=0, & \text{in}\,\, \Om \\
				v_j^l=f_l, & \text{on}\,\, \partial\Om.
			\end{array} \right.
		\end{equation*}
		Since solutions to this are unique, we define $v^l:=v_1^l=v_2^l.$
		
		For recovering the higher order derivatives of $c_j(x',0)$ we can use the second linearizations (from \eqref{secondlineqminimal2})
		\begin{equation}\label{secondlinproof}
			\left\{\begin{array}{ll}
				\dive \left(c_j(x',0)^{\frac{n-1}{2}}\nabla w_j^{(al)}\right)-\frac{n-1}{2}c_j(x',0)^{\frac{n-1}{2}-1}\partial_{x_n}^3c_j(x',0)v^lv^a=0, & \text{in}\,\, \Om \\
				w_j^{(al)}=0, & \text{on}\,\, \partial\Om. \phantom{\Big|}
			\end{array} \right.
		\end{equation}
		corresponding to $j=1,2$. Notice that if we replace $c_2(x',0)$ by $\lambda^{-1}c_1(x',0)$ in \eqref{secondlineqminimal1}, except in the third order derivative, we get that $w_2^{al}$ solves
		\begin{equation}\label{secondlinproof2}
			\left\{\begin{array}{ll}
				\dive \left(c_1(x',0)^{\frac{n-1}{2}}\nabla w_2^{(al)}\right) - \lambda\frac{n-1}{2}c_1(x',0)^{\frac{n-1}{2}-1}\partial_{x_n}^3c_2(x',0)v^lv^a=0, & \text{in}\,\, \Om \\
				w_2^{(al)}=0, & \text{on}\,\, \partial\Om. \phantom{\Big|}
			\end{array} \right.
		\end{equation}
		Subtract now \eqref{secondlinproof} for $j=1$ from \eqref{secondlinproof2}, integrate against $v\equiv1$ (solution to the first linearization) over $\Om$ 
		\begin{align*}
			&\int_{\Om}\dive (c_1(x',0)^{\frac{n-1}{2}}\nabla w_1^{(al)}-c_1(x',0)^{\frac{n-1}{2}}\nabla w_2^{(al)})\\
			&-\left(\frac{n-1}{2}c_1(x',0)^{\frac{n-1}{2}-1}\partial_{x_n}^3c_1(x',0)-\lambda\frac{n-1}{2}c_1(x',0)^{\frac{n-1}{2}-1}\partial_{x_n}^3c_2(x',0)\right)v^lv^a\,dx'\\
			&=0
		\end{align*}
		and use integration by parts to have
		\begin{align*}
			0&=\int_{\partial\Om}c_1(x',0)^{\frac{n-1}{2}}\left(\nabla w_1^{(al)}\cdot\nu-\nabla w_2^{(al)}\cdot\nu\right)\,dS\\
			&=\int_{\Om}\dive (c_1(x',0)^{\frac{n-1}{2}}\nabla w_1^{(al)}-c_2(x',0)^{\frac{n-1}{2}}\nabla w_2^{(al)})\,dx'\\
			&=\int_{\Om}\frac{n-1}{2}c_1(x',0)^{\frac{n-1}{2}-1}\left(\partial_{x_n}^3c_1(x',0)-\lambda\partial_{x_n}^3c_2(x',0)\right)v^lv^a\,dx'.
		\end{align*}
		This is true since by \eqref{DN equality}
		\begin{equation*}
			\partial_{\nu}u_1\big|_{\partial\Om}=\partial_{\nu}u_2\big|_{\partial\Om}
		\end{equation*}
		and applying $\partial_{\e_a}\partial_{\e_l}\big|_{\e=0}$ to this implies
		\begin{equation*}\label{w boundary}
			\partial_{\nu}w_1^{(al)}\big|_{\partial\Om}=\partial_{\nu}w_2^{(al)}\big|_{\partial\Om},\quad\text{for}\va a,l\in\{1,\ldots,k\}.
		\end{equation*}
		
		Thus
		\begin{equation}
			\int_{\Om}c_1(x',0)^{\frac{n-1}{2}-1}\left(\partial_{x_n}^3c_1(x',0)-\lambda\partial_{x_n}^3c_2(x',0)\right)v^lv^a\,dx'=0
		\end{equation}
		for any $v^a, v^l$ solving the conductivity equation \eqref{firstlinminimalproof}. A solution to \eqref{firstlinminimalproof} is equivalently a solution to 
		\begin{equation*}\label{change to schr}
			\left\{\begin{array}{ll}
				\left(-\Delta_{x'}+\frac{\Delta_{x'}c_j(x',0)^{\alpha/2}}{c_j(x',0)^{\alpha/2}}\right)g^l=0, & \text{in}\,\, \Om \\
				g^l=c_j(x',0)^{\alpha/2}f_l, & \text{on}\,\, \partial\Om, \phantom{\Big|}
			\end{array} \right.
		\end{equation*}
		where $\alpha=\frac{n-1}{2}$, $g^l=c_j(x',0)^{\alpha/2}v^l$ and $\Delta_{x'}$ denotes the Laplacian with respect to the first two variables. Hence by using the fact that a product of a pair of solutions (Proposition \ref{cgo}) is dense in $L^1(\Om)$, we get
		\begin{equation*}
			\partial_{x_n}^3c_1(x',0)=\lambda\partial_{x_n}^3c_2(x',0),\quad x'\in\Om.
		\end{equation*}
		Also \eqref{secondlinproof}, \eqref{secondlinproof2} together with the previous equality and $c_1(x',0)=\lambda c_2(x',0)$, gives the following boundary value problem
		\begin{equation*}
			\left\{\begin{array}{ll}
				\dive\left(c_1(x',0)^{\frac{n-1}{2}}\nabla\left(w_1^{(al)}-w_2^{(al)}\right)\right)=0, & \text{in}\,\, \Om \\
				w_1^{(al)}-w_2^{(al)}=0, & \text{on}\,\, \partial\Om. \phantom{\Big|}
			\end{array} \right.
		\end{equation*}
		This has a unique solution and thus $w_1^{(al)}=w_2^{(al)}$.


		Next we use induction to show $\partial_{x_n}^kc_1(x',0)=\lambda\partial_{x_n}^kc_2(x',0)$ for all $k\in\N$. By the above this already holds for $k=0, 1, 2, 3.$ Our assumption now is
		\begin{equation*}\label{induction}
			\partial_{x_n}^kc_1(x',0)=\lambda\partial_{x_n}^kc_2(x',0),\quad x'\in\Om,\va \text{for all}\quad k=0,1,2,\ldots,N\in\N, N>3.
		\end{equation*}
		Let us do a subinduction to prove
		\begin{equation*}
			\partial_{l_1\dots l_k}^ku_1(x',0)=\partial_{l_1\dots l_k}^ku_2(x',0),\quad x'\in\,\,\Om,
		\end{equation*}
		for all $k=1,\ldots,N$, where $\partial_{l_1\dots l_k}^ku_j(x',0)=\frac{\partial^ku_j(x',0)}{\partial_{\e_{l_1}}\dots\partial_{\e_{l_k}}}$. Above we have shown this for $k=1,2$. Assume that it holds for $k\leq K<N$. Then the linearization of order $K+1$ is, when evaluated at $\e_1=\dots=\e_{K+1}=0$,
		\begin{align}\label{subinduction}
			&\dive\left(c_j(x',0)^{\frac{n-1}{2}}\nabla\left(\partial_{l_1\dots l_{K+1}}^{K+1}u_j(x',0)\right)\right)+R_K(u_j,c_j(x',0),0)\\\notag
			&+Cc_j(x',0)^{\frac{n-1}{2}-1}\partial_{x_n}^{K+2}c_j(x',0)\left(\prod_{k=1}^{K+1}v^{(l_k)}\right)=0,
		\end{align}
		$x'\in\,\,\Om$, where $C\neq0$. Actually $C=\frac{n-1}{2}$, since it comes from the $u$ derivatives of $F$ and is the constant $\frac{n-1}{2}$ appearing in front of the second term of $F$. Also, here $R$ is a polynomial of $\partial_{x_n}^kc_j(x',0), \partial_{l_1\dots l_k}^ku_1(x',0)$ and the components of $\nabla_{x'}\left(\partial_{x_n}^kc_j(x',0)\right)$. Now an integration by parts argument similar to the case of the second linearization and together with Proposition \ref{cgo} (choosing $v^3=\ldots=v^{K+1}=1$) gives $\partial_{x_n}^{K+2}c_1(x',0)=\lambda\partial_{x_n}^{K+2}c_2(x',0)$.
		
		Subtracting the equations \eqref{subinduction} (similarly as for equations \eqref{secondlinproof} and\eqref{secondlinproof2}) for $j=1,2$ we get
		\begin{equation*}
			\left\{\begin{array}{ll}
				\dive\left(c_1(x',0)^{\frac{n-1}{2}}\nabla\left(\partial_{l_1\dots l_{K+1}}^{K+1}u_1(x',0)-\partial_{l_1\dots l_{K+1}}^{K+1}u_2(x',0)\right)\right)=0, & \text{in}\,\, \Om \\
				\partial_{l_1\dots l_{K+1}}^{K+1}u_1(x',0)-\partial_{l_1\dots l_{K+1}}^{K+1}u_2(x',0)=0, & \text{on}\,\, \partial\Om. \phantom{\Big|}
			\end{array} \right.
		\end{equation*}
		This is true, since by induction assumptions for all $x'\in\Om$ we have \linebreak $\nabla_{x'}\left(\partial_{x_n}^kc_1(x',0)-\partial_{x_n}^kc_2(x',0)\right)=0$ and the other terms agree for $j=1,2$, $k\leq K$. Again, by the uniqueness of solutions, $\partial_{l_1\dots l_{K+1}}^{K+1}u_1(x',0)=\partial_{l_1\dots l_{K+1}}^{K+1}u_2(x',0)$, $x'\in\Om$, which ends the subinduction.
		
		Returning to the original induction, the linearization of order $N+1$ at $\e_1=\dots=\e_{N+1}=0$ is
		\begin{align*}
			&\dive\left(c_j(x',0)^{\frac{n-1}{2}}\nabla\left(\partial_{l_1\dots l_{N+1}}^{N+1}u_j(x',0)\right)\right)\\
			&+R_{N+1}(u_j,c_j(x',0),0)+Cc_j(x',0)^{\frac{n-1}{2}-1}\partial_{x_n}^{N+2}c_j(x',0)\left(\prod_{k=1}^{N+1}v^{(l_k)}\right)=0,
		\end{align*}
		$x'\in\,\,\Om.$
		By the subinduction, the terms $R_N(u_j,c_j(x',0),0)$ agree for $j=1,2$. Thus by subtracting, using integration by parts and that $\partial_{\nu}\partial_{l_1\dots l_{N+1}}^{N+1}u_1(x',0)|_{\partial\Om}=\partial_{\nu}\partial_{l_1\dots l_{N+1}}^{N+1}u_2(x',0)|_{\partial\Om}$ we get
		\begin{equation*}
			\int_{\Om}c_j(x',0)^{\frac{n-1}{2}-1}\left(\partial_{x_n}^{N+2}c_1(x',0)-\lambda\partial_{x_n}^{N+2}c_2(x',0)\right)\prod_{k=1}^{N+1}v^{l_k}\,dx'=0.
		\end{equation*}
		Choosing all but two of the functions $v^{l_k}$ to be equal to $1$, we have by the completeness of such solutions (Proposition \ref{cgo}) that
		\begin{equation*}
			\partial_{x_n}^{N+2}c_1(x',0)=\lambda\partial_{x_n}^{N+2}c_2(x',0),\quad x'\in\,\,\Om,
		\end{equation*}
		which ends the proof for case $(1)$.
		
		\vspace{8pt}
		\textbf{Case $(2)$:} 
		Now we assume that the partial DN maps coincide for $f\in U_{\delta}$, $\spt(f)\subset\Gamma$. Then, as in the previous case, from the first linearization we get $c_1(x',0)=\lambda c_2(x',0)$ in $\Om$ (using first the boundary determination from \cite{Brown2006}, then \cite{Imanuvilov2010} for $n=3$ and \cite{Isakov2007} for $n>3$). Now define $v^l:=v^l_1=v^l_2$ again by uniqueness of solutions.
		
		Moving to the second order linearizations and recovering higher order derivatives produces some extra work since we only have partial data. From the assumption that the DN maps coincide we get
		\begin{equation}\label{(3) 1.}
			\partial_{\nu}w_1^{al}\big|_{\Gamma}=\partial_{\nu}w_2^{al}\big|_{\Gamma}.
		\end{equation}
		If we would now integrate the difference of \eqref{secondlinproof}, for $j=1$, and \eqref{secondlinproof2} against $v\equiv1$ and integrate by parts, some terms would not cancel out. Let us instead introduce the function $v^{(0)}$ which is a solution to
		\begin{equation}\label{nonzero solution}
			\left\{\begin{array}{ll}
				\dive\left(c_1(x',0)^{\frac{n-1}{2}} \nabla v^{(0)}\right)=0, & \text{in}\,\, \Om \\
				v^{(0)}=0, & \text{on}\,\, \partial\Om\setminus\Gamma \phantom{\Big|}\\
				v^{(0)}=g, & \text{on}\,\, \Gamma,
			\end{array} \right.
		\end{equation}
		where $g\in C^{\infty}_c(\Gamma)$ such that $g\geq0$ and $g\not\equiv0$. Then by the maximum principle $v^{(0)}>0$ in $\Om$. Now we integrate against this and use integration by parts to have
		\begin{align*}
			&\int_{\Om}\frac{n-1}{2}c_1(x',0)^{\frac{n-1}{2}-1}\left(\partial_{x_n}^3c_1(x',0)-\lambda\partial_{x_n}^3c_2(x',0)\right)v^{(0)}v^lv^a\,dx'\\
			&=\int_{\Om}\dive \left(c_1(x',0)^{\frac{n-1}{2}}\nabla \left(w_1^{(al)}- w_2^{(al)}\right)\right)v^{(0)}\,dx'\\
			&=\int_{\Om}(w_1^{al}-w_2^{al})\dive\left(c_1(x',0)^{\frac{n-1}{2}} \nabla v^{(0)}\right)\,dx'\\
			&+\int_{\Gamma}c_1(x',0)^{\frac{n-1}{2}}\left(\partial_{\nu}(w_1^{al}-w_2^{al})v^{(0)}-(w_1^{al}-w_2^{al})\partial_{\nu}v^{(0)}\right)\,dS\\
			&+\int_{\partial\Om\setminus\Gamma}c_1(x',0)^{\frac{n-1}{2}}\left(\partial_{\nu}(w_1^{al}-w_2^{al})v^{(0)}-(w_1^{al}-w_2^{al})\partial_{\nu}v^{(0)}\right)\,dS\\
			&=0.
		\end{align*}
		In the last inequality we used equation \eqref{(3) 1.}, the fact that $v^{(0)}$ solves \eqref{nonzero solution} and that $w_j^{al}=0$ on $\partial\Om$ for $j=1,2$. Then using Proposition \ref{cgo} and the positivity of $v^{(0)}$ we can conclude
		\begin{equation*}
			\partial_{x_n}^3c_1(x',0)=\lambda\partial_{x_n}^3c_2(x',0),\quad x'\in\Om.
		\end{equation*}
		
		As in the previous case we use induction to show $\partial_{x_n}^kc_1(x',0)=\lambda\partial_{x_n}^kc_2(x',0)$ for all $k\in\N$. By the above this already holds for $k=0, 1, 2, 3.$ Our assumption now is
		\begin{equation*}\label{induction2}
			\partial_{x_n}^kc_1(x',0)=\lambda\partial_{x_n}^kc_2(x',0),\quad x'\in\Om,\va \text{for all}\quad k=0,1,2,\ldots,N\in\N, N>3.
		\end{equation*}
		By a subinduction we can show that 
		\begin{equation*}
			\partial_{l_1\dots l_k}^ku_1(x',0)=\partial_{l_1\dots l_k}^ku_2(x',0),\quad x'\in\,\,\Om,
		\end{equation*}
		for all $k=1,\ldots,N$, where $\partial_{l_1\dots l_k}^ku_j(x',0)=\frac{\partial^ku_j(x',0)}{\partial_{\e_{l_1}}\dots\partial_{\e_{l_k}}}$. This goes in the same way as in the previous case except the integration by parts argument needs to be done as shown in this case.
		
		Returnig to the original induction, the rest of the proof is again the same as in case $(2)$. We only need to modify the integration by parts argument using again the function $v^{(0)}$ and Proposition \ref{cgo} which finishes the proof.
	\end{proof}

	\printbibliography
\end{document}